\documentclass{article} 
\usepackage{latexsym}
\usepackage{graphicx,subfigure}
\usepackage{array}
\usepackage{indentfirst}
\usepackage{fancyhdr}
\usepackage{epsf}
\usepackage{amsmath,bm}
\usepackage{amsthm}
\usepackage{booktabs}
\usepackage{amsfonts}
\usepackage{epsf}
\usepackage{multicol}
\usepackage[square, comma, sort&compress, numbers]{natbib}
\usepackage{fancyhdr}
\usepackage{pxfonts}
\usepackage{caption}
\usepackage{epstopdf}
\usepackage{algorithm}
\usepackage{algpseudocode}
\usepackage{exscale}
\usepackage{relsize}
\usepackage{setspace}
\usepackage{geometry}  
\usepackage{xcolor} 
\usepackage{url}

\newcommand{\bc}{\begin{center}}
\newcommand{\ec}{\end{center}}
\newcommand{\be}{\begin{eqnarray}}
\newcommand{\ee}{\end{eqnarray}}

\newcommand{\ben}{\begin{eqnarray*}}
\newcommand{\een}{\end{eqnarray*}}

\newcommand{\Om}{{\rm\Omega}}

\newcommand{\dx}{\,dx}

\newcommand{\Rmnum}[1]{\expandafter\@slowromancap\romannumeral #1@}

\newcommand{\cE}{{\mathcal{E}_h}}
\newcommand{\cEi}{{\mathcal{E}_h^I}}

\newcommand{\cED}{{\mathcal{E}_h^D}}
\newcommand{\cEN}{{\mathcal{E}_h^N}}

\newcommand{\cTh}{{\mathcal{T}}}
\newcommand{\PcT}{{\partial\mathcal{T}_h}}
\newcommand{\R}{\mathbb{R}}
\def\S{\mathcal{S}}

\newcommand{\Q}{\boldsymbol{Q}}
\newcommand{\checkQ}{\boldsymbol{\check{Q}}}

\newcommand{\V}{\boldsymbol{V}}
\newcommand{\checkV}{\boldsymbol{\check{V}}}
\newcommand{\Qe}{\boldsymbol{\Sigma}}
\newcommand{\Z}{\boldsymbol{Z}}
\newcommand{\checkQe}{\boldsymbol{\check{\Sigma}}}
\newcommand{\Hs}{\boldsymbol{H}}

\def\undertilde#1{\mathord{\vtop{\ialign{##\crcr
$\hfil\displaystyle{#1}\hfil$\crcr\noalign{\kern1.5pt\nointerlineskip}
$\hfil\tilde{}\hfil$\crcr\noalign{\kern1.5pt}}}}}

\newcommand{\cp}{\boldsymbol{p}}  
  
\newcommand{\dph}{\boldsymbol{p}_h}
\newcommand{\dqh}{\boldsymbol{q}_h}
\newcommand{\du}{u_h}
\newcommand{\dv}{v_h}
\newcommand{\hp}{\hat{\boldsymbol{p}}_h}
\newcommand{\ttp}{\tilde{\boldsymbol{p}}_h}
\newcommand{\tp}{\check{\boldsymbol{p}}_h}
\newcommand{\tq}{\check{\boldsymbol{q}}_h}
\newcommand{\tu}{\check{u}_h}
\newcommand{\tv}{\check{v}_h}

\newcommand{\cu}{\boldsymbol{u}}
\newcommand{\cv}{\boldsymbol{v}}
\newcommand{\csig}{\boldsymbol{\sigma}}
\newcommand{\ctau}{\boldsymbol{\tau}}
\newcommand{\due}{\boldsymbol{u}_h}
\newcommand{\dve}{\boldsymbol{v}_h}
\newcommand{\dsig}{\boldsymbol{\sigma}_h}
\newcommand{\dtau}{\boldsymbol{\tau}_h}
\newcommand{\hsig}{\check{\boldsymbol{\sigma}}_h}
\newcommand{\htau}{\check{\boldsymbol{\tau}}_h}
\newcommand{\hue}{\check{\boldsymbol{u}}_h}
\newcommand{\hve}{\check{\boldsymbol{v}}_h}
\newcommand{\tsig}{\hat{\boldsymbol{\sigma}}_h} 
\newcommand{\tue}{\hat{\boldsymbol{u}}_h}
\newcommand{\phie}{\boldsymbol{\phi}}  
\newcommand{\psie}{\boldsymbol{\psi}}  

\newcommand{\dw}{\boldsymbol{w}_h} 
\newcommand{\Piu}{{\rm P}_h}
\newcommand{\Piut}{{\rm \tilde P}_h}
\newcommand{\Pipe}{\boldsymbol{\Pi}_h}
\newcommand{\Piue}{{\textbf{P}}_h}

\newcommand{\ep}{\boldsymbol{d}_h}

\newcommand{\epu}{\boldsymbol{e}_{\cu}}

\newcommand{\divh}{{\rm div}_h }
\newcommand{\bn}{\boldsymbol{n}}
\newcommand{\normmm}[1]{{\left\vert\kern-0.25ex\left\vert\kern-0.25ex\left\vert #1
    \right\vert\kern-0.25ex\right\vert\kern-0.25ex\right\vert}}

\newtheorem{remark}{Remark}[section]

\newtheorem{theorem}{Theorem}[section]
\newtheorem{lemma}[theorem]{Lemma}

\newtheorem{Ass}{Assumption}[section]
\theoremstyle{definition}
\usepackage[titletoc]{appendix}
\numberwithin{equation}{section}

\begin{document}
\title{Superconvergence of Discontinuous Galerkin methods for Elliptic Boundary Value Problems
}
\date{\vspace{-5ex}}
%
\author{\textsc{ Limin Ma}\thanks{
Department of Mathematics, Pennsylvania State University, University Park, PA, 16802, USA. maliminpku@gmail.com
}
}

\maketitle
\begin{abstract}
In this paper, we present a unified analysis of the superconvergence property for a large class of mixed discontinuous Galerkin methods. This analysis applies to both the Poisson equation and linear elasticity problems with symmetric stress formulations. Based on this result, some locally postprocess schemes are employed to improve the accuracy of displacement by order $\min(k+1, 2)$ if polynomials of degree $k$ are employed for displacement.  Some numerical experiments are carried out to validate the theoretical results.

 \vskip 15pt

\noindent{\bf Keywords. }{superconvergence, postprocessing, discontinuous Galerkin, linear elasticity problem}

\end{abstract}

\section{Introduction and Notation}
\subsection{Introduction}
In this work, we investigate the superconvergence property and postprocess schemes of mixed discontinuous Galerkin methods for two classes of problems. One is the second order model problem
\begin{equation}\label{model:elliptic}
\left\{\begin{aligned}
c\cp - \nabla u&=0& & \text { in } \Omega, \\
\operatorname{div}\cp &=f & & \text { in } \Omega, \\
u &=0 & & \text { on } \Gamma_{D}, \\
\cp \cdot \boldsymbol{n} &=0 & & \text { on } \Gamma_{N},
\end{aligned}\right.
\end{equation}
with $\Om\subset \R^n~(n=2,3)$ and $\partial \Om=\Gamma_D\cup \Gamma_N$, $\Gamma_D\cap \Gamma_N=\emptyset$. Here $c$ is a bounded and positive definite matrix from $\mathbb{R}^n$ to $\mathbb{R}^n$, $u$ is a scalar function and $\cp$ is a vector-valued function.
The other one is the linear elasticity problem
\begin{equation}\label{model:elasticity}
\left\{
\begin{aligned}[rll]
A\csig-\epsilon (\cu)&=0\quad &\text{ in }\Om ,\\
\operatorname{div} \csig&=f\quad & \text{ in }\Om ,\\
\cu&=0\quad &\text{ on }\Gamma_D ,\\
\csig\bn&=0\quad &\text{ on }\Gamma_N,
\end{aligned}
\right.
\end{equation}
with $\Om\subset \R^n~(n=2,3)$ and $\partial \Om=\Gamma_D\cup \Gamma_N$, $\Gamma_D\cap \Gamma_N=\emptyset$. 
 Here the displacement is denoted by $\cu: \Om\rightarrow \R^n$ and the stress tensor is denoted by $\csig: \Om\rightarrow \S$, where $\S$ is the set of symmetric $n\times n$ tensors. The linearized strain tensor $\epsilon (\cu)=\frac{1}{2}(\nabla \cu+\nabla \cu^T)$. The compliance tensor $A: \S\rightarrow \S$
 \begin{equation}\label{lame}
A\csig ={1+\nu\over E}\csig -{\nu\over E}tr(\csig)I
\end{equation}
is assumed to be bounded and symmetric positive definite, where $E$ and $\nu$ are the Young's modulus and Poisson ratio of the elastic material under consideration, respectively.  

Postprocessing type of superconvergence property was discussed in literature, see \cite{bramble1989local,xu2003super,arnold1985mixed,stenberg1991postprocessing} for instance. There are two main ingredients for this kind of superconvergence. One is the superclose property of the projection of the exact solution, and the other one is an appropriate postprocess scheme  which is performed seperately on each element.  
For the scalar elliptic problem \eqref{model:elliptic}, the postprocessing type of superconvergence has been analyzed for the conforming elements, mixed element and nonconforming elements with superclose property, 
see  for instance, \cite{arnold1985mixed,bank2019superconvergent,chenhigh,Chen2013Superconvergence,Brandts1994Superconvergence,Jan2000Superconvergence,Douglas1989Superconvergence,hong2018uniform,stenberg1991postprocessing,hong2012discontinuous,hong2019unified} 
and the references therein.
For some nonconforming elements, the lack of  this superclose property of the canonical interpolation leads to the difficulty in analyzing the superconvergence result.
Recently, a superconvergence of two nonconforming elements in this case was analyzed in \cite{Hu2016Superconvergence,hu2018optimal} by employing the superclose property of a related mixed element. 
The superconvergence property was also analyzed for 
various discontinuous Galerkin methods \cite{cockburn2009superconvergent,cockburn2012conditions,xie2009numerical}.
 For the linear elasticity problem \eqref{model:elasticity}, the strong symmetry of the stress tensor causes a substantial additional difficulty for developing stable mixed elements for elasticity problem \cite{arnold2002mixed,arnold2008finite,hu2014family,hu2015family,hu2016finite,hu2014finite,hong2016robust}.  
The mixed methods in \cite{cockburn2010new, gopalakrishnan2012second} and hybridizable  discontinuous Galerkin  methods in \cite{cockburn2013superconvergent} imposed weak symmetry on the stress tensor, and achieved optimal convergence for stress and superconvergence for displacement by post processing. 
A postprocessing schemes was analyzed for a mixed element methods solving the linear elasticity problems   \eqref{model:elasticity} in \cite{stenberg1988family}. A superconvergent hybridizable  discontinuous Galerkin method with strong symmetry was analyzed in \cite{cockburndevising}.

In this paper, a unified superconvergence analysis of a large class of 
mixed discontinuous Galerkin methods 
is presented for both the scalar elliptic problem \eqref{model:elliptic} and linear elasticity problem \eqref{model:elasticity} in \cite{XGHong,hong2020extended} . 
Mixed discontinuous Galerkin methods employ discontinuous polynomials with degree $k$ and $k+1$ for the displacement $\cu$ and the stress $\csig$, respectively.
Thanks to a conforming projection and the corresponding commuting diagram, the $L^2$ projections of $u$  for \eqref{model:elliptic} and  $\cu$ for \eqref{model:elasticity} admit a superclose property. 
Note that this property can be advantageously exploited  to design a high accuracy approximation to $u$ and $\cu$. Indeed, following the idea in \cite{arnold1985mixed,Douglas1989Superconvergence,stenberg1991postprocessing,cockburn2009superconvergent,cockburn2012conditions,xie2009numerical}, we propose four postprocessing schemes  for the mixed discontinuous Galerkin method in \cite{XGHong} and get new approximations to $u$ with high accuracy for second order scalar elliptic problem \eqref{model:elliptic}.  For some special choices of parameters, the mixed discontinuous Galerkin method in \cite{XGHong} is hybridizable and leads to a much smaller system. 
The variable $\hp$ in the hybridized formulation is an approximation to $\cp$ on edges.
This $\hp$, together with the aforementioned postprocessing scheme, gives rise to a superconvergent approximation to the solution $u$ of \eqref{model:elliptic}. 
For the elasticity problem \eqref{model:elasticity}, 
a post processing scheme in \cite{stenberg1988family} was analyzed  for a mixed element method. In this paper, a similar  scheme is proposed for the discontinuous Galerkin method with symmetric stress in \cite{hong2020extended}. 
The proposed postprocessing scheme is analyzed to admit a desirable superconvergence property when $k\ge n$, which improves the accuracy of displacement by order $\min(k+1, 2)$ if polynomials of degree $k$ is employed for displacement. 
The current result 
provides the first analysis for a number of new methods  \cite{cockburn2009unified,hong2019unified,MixLDGWu,hong2020extended}. The numerical tests for linear elasticity problems also indicate that there is no such conforming interpolation which  admits the commuting diagram when $k<n$.

The rest of the paper is organized as follows. 
Section \ref{sec:elliptic} and \ref{sec:elasticity} analyze the postprocessing schemes and the  superconvergence property for scalar elliptic problems and linear elasticity problems, respectively.   
Some numerical examples are tested in Section \ref{sec:num} to verify the theoretical results.

\subsection{Notation}

Given a nonnegative integer $m$ and a bounded domain $D\subset \mathbb{R}^n$, let $H^m(D)$, $\|\cdot\|_{m,D}$ and $|\cdot|_{m,D}$ be the usual Sobolev space, norm
and semi-norm,  respectively.  The
$L^2$-inner product on $D$ and $\partial D$ are denoted by $(\cdot,
\cdot)_{D}$ and $\langle\cdot, \cdot\rangle_{\partial D}$,
respectively. Let  $\|\cdot\|_{0,D}$ and $\|\cdot\|_{0,\partial D}$ be the norms of  $L^2(D)$ and $L^2(\partial D)$, respectively. The norms $\|\cdot\|_{m,D}$ and $|\cdot|_{m,D}$ are abbreviated as  $\|\cdot\|_{m}$ and
$|\cdot|_{m}$, respectively, when $D$ is chosen as $\Omega$. Suppose that $\Om\subset \mathbb{R}^n$ is a bounded polygonal domain covered exactly by a shape-regular partition $\cTh$ into polyhedrons.  Let  $h_K$ be the diameter of element $K\in \cTh$ and $h=\max_{K\in\cTh}h_K$.  Denote the set of all interior edges/faces of $\cTh$ by $\cEi$, and all edges/faces on boundary $\Gamma_D$ and $\Gamma_N$ by $\cED$ and $\cEN$, respectively.  Let $\cE=\cEi\cup \cED \cup \cEN$ and $h_e$ be the diameter of edge/face $e\in \cE$. For any interior edge/face $e=K^+\cap K^-$, let $\bn^i$ = $\bn|_{\partial K^i}$ be the unit outward normal vector on $\partial K^i$ with $i = +,-$.  
For $K\subset\R^n$ and any nonnegative integer $r$, let $P_r(K, \R)$ be the space of all polynomials of degree not greater than $r$ on $K$.

Throughout this paper, we shall use letter $C$, which is independent
of mesh-size $h$, stabilization parameters $\eta, \tau, \gamma$, 
to denote a generic
positive constant which may stand for different values at different
occurrences.  Following \cite{xu1992iterative}, the notations $x \lesssim y$ and $x \gtrsim y$ mean $x
\leq Cy$  and $x \geq Cy$, respectively. Denote $x\lesssim y\lesssim x$ by $x \cong y$.

\section{Scalar elliptic problems}\label{sec:elliptic}
This section analyzes the postprocessing schemes and the superconvergence result for the scalar elliptic problem \eqref{XGdiv}.

\subsection{Discontinuous Galerkin methods for scalar elliptic problems}
Consider the second order elliptic model problem \eqref{model:elliptic}. 
For any scalar-valued function $\dv$ and vector-valued function $\dqh$ that are piecewise smooth with respect to $\mathcal{T}$, let $\dv^{\pm}$ = $\dv|_{\partial K^{\pm}}$, $\dqh^{\pm}$ = $\dqh|_{\partial K^{\pm}}$.  Define the average $\{\cdot\}$ and the jump $[\cdot ]$ on interior edges/faces $e\in \cEi$ as
follows: 
\begin{equation}\label{jumpdef}
\begin{array}{ll}
\{\dqh\}=\frac{1}{2}(\dqh^++\dqh^-),&[\dqh]=\dqh^+\cdot \bn^++\dqh^-\cdot \bn^-,\\
\{\dv\}=\frac{1}{2}(\dv^++\dv^-),&[\dv] =\dv^+\bn^+  + \dv^-\bn^-.
\end{array}
\end{equation} 
For any boundary edge/face $e\subset \partial \Omega$, define
\begin{equation}\label{bddef}
\begin{array}{lllll}
\{\dqh\}=\dqh, &  [\dqh]=0, &\{\dv\}=\dv,&[\dv]=\dv\bn,& \text{on }\Gamma_D,\\
\{\dqh\}=\dqh,&  [\dqh]=\dqh\cdot\bn, &\{\dv\}=\dv, & [\dv]=0,& \text{on }\Gamma_N.
\end{array}
\end{equation}   
For any scalar-valued function $\dv$ and vector-valued function $\dqh$,  define the piecewise gradient $\nabla_h$ and piecewise divergence $\divh$ by
$$
\nabla_h \dv\big |_K=\nabla (\dv|_K), \quad
\divh \dqh\big |_K={\rm div} (\dqh |_K) \quad \forall K \in \cTh.
$$
Define some inner products as follows:
\begin{equation} \label{equ:inner-product}
(\cdot,\cdot)_\cTh=\sum_{K\in \cTh }(\cdot,\cdot)_{K},
\quad \langle\cdot,\cdot\rangle =\sum_{e\in \cE}\langle\cdot,\cdot\rangle_{e},
\quad \langle\cdot,\cdot\rangle_{\partial\cTh }=\sum_{K\in\cTh}\langle\cdot,\cdot\rangle_{\partial K}.
\end{equation} 
Whenever there is no ambiguity, we simplify $(\cdot, \cdot)_\cTh$ as $(\cdot,\cdot)$. 
With the aforementioned definitions, the following DG identity \cite{arnold2002unified} holds:
\begin{equation}\label{DGidentity}
(\dqh, \nabla_h \dv)
=-(\divh \dqh, \dv)
+ \langle [\dqh], \{\dv\}\rangle
+ \langle \{\dqh\}, [\dv]\rangle.
\end{equation}
 The four-field extended Galerkin formulation in \cite{XGHong} seeks  $(\dph, \tp, \du, \tu)\in \Q_h\times  \checkQ_h\times V_h\times \check V_{h}$ such that 
\begin{equation}\label{XGdiv}
\left\{
\begin{array}{rll}
(c\dph,\dqh) 
+(\du, \divh \dqh) 
-\langle \{u_{h}\} +\tu - \gamma [u_h], [\dqh] \rangle
&=0,&\ \forall \dqh \in \Q_h,\\
-(\divh \dph, \dv) 
-\langle \gamma[\dph]+\tp, [\dv]\rangle 
+\langle [ \dph], \{\dv\}\rangle
&=-(f,\dv) &\ \forall \dv\in V_h,\\
-\langle \tau^{-1}   
\tp+ [\du], \tq \rangle_{e }
&=0,&\forall  \tq \in \checkQ_{h},
\\
\langle \eta^{-1} \tu + [\dph], \tv \rangle_e
&=0,&\forall  \tv \in \check{V}_{h},
\end{array}
\right.
\end{equation}  
where  
\begin{align*}
\Q_h:&=\{\dqh\in L^2(\Omega, \R^n): \dqh|_K\in \Q(K), \ \forall K\in \cTh_h\},\\
\checkQ_h:&=\{\tq\in L^2(\cE, \R^n): \dqh|_K\in \checkQ(K), \ \forall K\in \cTh_h\},\\
V_h:&=\{v_h\in L^2(\Omega, \R): \dqh|_K\in V(K), \ \forall K\in \cTh_h\},\\
\check{V}_h:&=\{\tv\in L^2(\cE, \R): \dqh|_K\in \check{V}(K), \ \forall K\in \cTh_h\}.
\end{align*}
Here $\gamma$ is constant, $\tp$ and $\tu$ are the modifications to $\dph$ and $\du$ on elementary boundaries, respectively. 
Define the discontinuous spaces $\Q_h$, $\checkQ_h$, $V_h$ and $\check V_h$ with
$$
\Q(K)=P_k(K,\R^n),\ \checkQ(K)=P_k(K,\R^n),\ V(K)=P_k(K,\R), \ \check V(K)=P_k(K,\R)
$$
by $\Q_h^{k}$, $\checkQ_h^k$, $V_h^k$ and $\check V_h^{k}$, respectively.
Define
\begin{equation}\label{normdef}
\begin{array}{ll}
\|\dqh\|_{\rm div, h}^2 =(c\dqh,\dqh)
+ \|\divh  \dqh\|_0^2
+ \| \eta^{1/2}[\dqh ]\|_0^2,
&
\|\tq\|_{0, h}^2 =\|\tau^{-1/2}\tq\|_0^2,\\
\|\dv\|_{0, h}^2 =\|\dv\|_0^2+ \|\tau^{1/2}[\dv]\|_0^2+\|\eta^{-1/2}\{\dv\}\|_0^2,
&
\|\tv\|_{0, h}^2 =\|\eta^{-1/2}\tv\|_0^2.
\end{array}
\end{equation} 

For $H({\rm div})$-based formulations \eqref{XGdiv}, the well-posedness and the error estimate  is analyzed in \cite{XGHong} under a set of assumptions as presented below. The error  estimate of $\dph$ in $L^2$-norm is similar to the one for the stress tensor in \cite{hong2020extended}, thus the details of the proof is omitted here.

\begin{lemma}\label{lm:elliptic}
For $H({\rm div})$-based four-field formulation \eqref{XGdiv} with $\eta=\left(\rho h_{e}\right)^{-1}, \tau \cong \eta^{-1}=\rho h_{e}$, if the spaces $\Q_{h}, V_{h}, \check V_{h}$ satisfy the conditions
\begin{enumerate}
\item[(C1)]  Let $\boldsymbol{R}_{h}:=\Q_{h} \cap \Hs(\operatorname{div}, \Omega)$ and $\boldsymbol{R}_{h} \times V_{h}$ is a stable pair for mixed method;
\item[(C2)]  $\divh \Q_{h}=V_{h}$
\item[(C3)]  $\left\{V_{h}\right\} \subset \check{V}_{h}$
\end{enumerate}
Then,  the formulation \eqref{XGdiv} is uniformly well-posed with respect to the norms \eqref{normdef} when $\rho\in (0, \rho_0]$. Namely, if $(\dph, \tp, \du, \tu)\in \Q_h\times  \boldsymbol{\check{Q}}_h\times V_h\times \check V_{h}$ is the solution of \eqref{XGdiv}, it holds that
$$
\| \dph\|_{\rm div, h} + \|\tp\|_{0, h} + \|\du\|_{0, h} + \|\tu\|_{0, h}\lesssim \|f\|_{0, \Omega}.
$$
If $\cp\in \Hs^{k+2}(\Omega, \R^n)$, $u\in H^{k+1}(\Omega, \R) (k\ge 0)$, and $\Q\times \checkQ_h\times V_h\times \check V_h=\Q_h^{k+1} \times \checkQ_h^k\times V_h^k\times \check V_h^{k+1}$,%
\begin{equation}\label{XGerr}
\| \cp - \dph\|_{\rm div, h} + \|\tp\|_{0, h} + \|u-\du\|_{0, h} + \|\tu\|_{0, h}\lesssim h^{k+1}(|\cp|_{k+2} + |u|_{k+1}).
\end{equation}
Furthermore, if $\cp\in \Hs^{k+2}(\Omega, \R^n)$,
\begin{equation}\label{err:L2}
\| \cp - \dph\|_0\lesssim h^{k+2}(|\cp|_{k+2} + |u|_{k+1}).
\end{equation}
\end{lemma}

We can establish the following superclose property for the extended Galerkin formulation  \eqref{XGdiv}. 
\begin{theorem}\label{th:super1}
Suppose $\cp\in H^{k+2}(\Omega, \R^n)$, $u\in H^{k+1}(\Omega, \R) (k\ge 0)$, and  
$(\dph, \tp, \du, \tu)\in \Q_h^{k+1}\times  \checkQ_h^k\times V_h^k\times \check V_h^{k+1}$ is the solution of the four-field formulation \eqref{XGdiv} with $\eta=\left(\rho h_{e}\right)^{-1}, \tau \cong \eta^{-1}=\rho h_{e}$. It holds that
\begin{equation*}
\|\Piu^k u - \du \|_{0, \Omega}\lesssim h^{\min(2k+2, k+3)}(|\cp|_{k+2} + |u|_{k+1}),
\end{equation*}
where $\Piu^k$ is the $L^2$-projection onto $V_h^{k}$.
\end{theorem}
We omit the proof here since it is similar to the analysis for linear elasticity problems \eqref{model:elasticity} in the next section. 

\subsection{Postprocess techniques for scalar elliptic problems}
Consider the $H({\rm div})$-based four-field formulation \eqref{XGdiv} with $\Q_h=\Q_h^{k+1}$, $\checkQ_h=\checkQ_h^k$, $V_h= V_h^k$ and $\check V_h= \check V_h^{k+1}$.
Define 
\begin{equation}\label{hatdef}
\hp =  \{\dph\} + \gamma[\dph]+\tp.
\end{equation}
Note that $\hp$ is an approximation to $\cp$ on elementary boundaries. 

We list three postprocessing techniques \cite{arnold1985mixed,gastaldi1989sharp,stenberg1988family,stenberg1991postprocessing,cockburn2012conditions,cockburn2009superconvergent} for the elliptic problem \eqref{model:elliptic}. Here there are two choices of the projection operator $\Piu$, one is the $L^2$  projection to piecewise constant space, namely
$$
\Piu u= \Piu^0u,
$$
where $\Piu^k$ is the $L^2$  projection to $V_h^k$,
and the other one is the $L^2$  projection to the discrete space $V_h$, namely
$$
\int_\Om \Piu u \dv \dx = \int_\Om u\dv\dx,\quad \forall \dv\in V_h.
$$
For either choice of $\Piu$, consider the following three postprocessing schemes:
\begin{enumerate}
\item Let $u_{1, h}^*\in V_h^{k+2}$ be the solution of  
\begin{equation}\label{postdef11}
\left\{
\begin{aligned}
\int_K  \alpha\nabla u_{1, h}^*\cdot \nabla \dv dx&= - \int_K f\dv dx + \int_{\partial K}\dph\cdot \bn  \dv ds,&\forall \dv\in (I - \Piu)V_h^{k+2}\big |_K,
\\
\Piu (u_{1, h}^* - \du)&=0.
\end{aligned}
\right.
\end{equation}
with $\alpha = c^{-1}$.
\item Let $u_{2, h}^*\in V_h^{k+2}$ be the solution of 
\begin{equation}\label{postdef12}
\left\{
\begin{aligned}
\int_K  \alpha\nabla u_{2, h}^*\cdot \nabla \dv dx&= - \int_K f\dv dx + \int_{\partial K}\hp\cdot \bn  \dv ds,&\forall \dv\in (I - \Piu)V_h^{k+2}\big |_K,
\\
\Piu (u_{2, h}^* - \du)&=0.
\end{aligned}
\right.
\end{equation}
with $\alpha = c^{-1}$ and  $\hp$ defined in \eqref{hatdef}.
\item Let $u_{3, h}^*\in V_h^{k+2}$ be the solution of 
\begin{equation}\label{postdefe2}
\left\{
\begin{aligned}
\int_K  \nabla u_{3, h}^*\cdot \nabla \dv dx&= \int_K c\dph\cdot \nabla \dv dx,&\forall \dv\in (I - \Piu^0)V_h^{k+2}\big |_K,
\\
\Piu^0 (u_{3, h}^* - \du)&=0.
\end{aligned}
\right.
\end{equation}  
\end{enumerate}
Note that  the schemes  \eqref{postdef12} and \eqref{postdefe2} are identical in some special cases.
If $\alpha$ is a constant matrix, the first equation in  \eqref{postdefe2} is equivalent to 
$$
\int_K  \alpha\nabla u_{3, h}^*\cdot \nabla \dv dx= \int_K \dph \cdot \nabla \dv dx
= -\int_K \divh \dph \dv dx + \int_{\partial K}\dph \bn \dv ds
$$
for any $\dv\in (I - \Piu^0)V_h^{k+2}\big |_K$.
By \eqref{XGdiv}, the above equation reads
$$
\int_K  \alpha\nabla u_{3, h}^*\cdot \nabla \dv dx= -\int_K f\dv dx + \int_{\partial K} \hp\bn \dv ds.
$$
It implies that for this particular $\alpha$, the postprocess algorithms \eqref{postdef12} with 
$\Piu=\Piu^0$
 and \eqref{postdefe2} are the same.

We analyze in the following theorem that the above postprocessing techniques can improve the accuracy for the mixed discontinuous Galerkin formulation \eqref{XGdiv}.
\begin{theorem}\label{th:post1}
Suppose $\cp\in \Hs^{k+2}(\Omega)$, $u\in H^{k+3}(\Omega) (k\ge 0)$, and  
$(\dph, \tp, \du, \tu)\in \Q_h^{k+1}\times  \checkQ_h^k\times V_h^k\times \check V_h^{k+1}$ is the solution of the four-field formulation \eqref{XGdiv} with $\eta=\left(\rho h_{e}\right)^{-1}, \tau \cong \eta^{-1}=\rho h_{e}$. It holds that
\begin{equation*} 
\|u-\du^*\|_0 \lesssim h^{\min(2k+2, k+3)}|u|_{k+3},
\end{equation*} 
where $\du^*=u_{1, h}^*$ in \eqref{postdef11}, $u_{2, h}^*$ in \eqref{postdef12} or $u_{3, h}^*$ in \eqref{postdefe2}.
\end{theorem}
\begin{proof}
Let $\dv= (I - \Piu)(\Piu^{k+2} u - \du^*)$. Since 
$\Piu^0\dv=0$,
\begin{equation}\label{post1}
\begin{aligned}
\|\dv\|_{0, \Om}=\|\dv - \Piu^0\dv\|_{0, \Om}&\lesssim h|\dv|_{1, h}.
\end{aligned}
\end{equation}
It follows from the trace inequality that
\begin{equation}\label{err:vedge}
\begin{aligned}
\|[\dv]\|_\cE + \|\{\dv\}\|_\cE&\lesssim h^{-1/2}\|\dv\|_{0, \Omega}\lesssim h^{1/2}|\dv|_{1, h}.
\end{aligned}
\end{equation}
By \eqref{postdef11} and \eqref{postdef12},
\begin{equation*}
\begin{aligned}
|\dv|_{1, h}^2=&(\alpha \nabla_h (\Piu^{k+2} u - \du^*), \nabla_h \dv) 
- (\alpha \nabla_h \Piu(\Piu^{k+2} u - \du^*), \nabla_h \dv) 
\\
=&(\alpha \nabla_h \Piu^{k+2} u, \nabla_h \dv) + (f, \dv) - \langle [\ttp],  \{\dv\} \rangle - \langle \{\ttp\},  [\dv] \rangle
\\
&- (\alpha \nabla_h \Piu(\Piu^{k+2} u - \du^*), \nabla_h \dv),
\end{aligned}
\end{equation*}
where $\ttp=\dph$ if $\du^*=u_{1, h}^*$, and $\ttp=\hp$ if $\du^*=u_{2, h}^*$.
Since $f=\nabla\cdot (\alpha \nabla u)$ and $\cp = \alpha \nabla u$,
\begin{equation}\label{eq:vH1}
\begin{aligned}
|\dv|_{1, h}^2
=&(\alpha \nabla_h (\Piu^{k+2} - I)u, \nabla \dv) - \langle [\ttp],  \{\dv\} \rangle + \langle \{\cp - \ttp\},  [\dv] \rangle
\\
&- (\alpha \nabla_h \Piu(\Piu^{k+2} u - \du^*), \nabla_h \dv).
\end{aligned}
\end{equation}
If $\Piu=\Piu^0$, the last term on the right hand side of the above equation equals zero. If $\Piu$ is the $L^2$ projection to $V_h$, namely $\Piu=\Piu^k$, 
by the triangle inequality and the inverse estimate,
$$
|(\alpha \nabla_h \Piu(\Piu^{k+2} u - \du^*), \nabla_h \dv)|\lesssim h^{-1}\|\Piu(\Piu^{k+2} u - \du^*)\|_0|\dv|_{1,h}.
$$
Since $\Piu \du^*=\Piu\du$,
$$
|(\alpha \nabla_h \Piu(\Piu^{k+2} u - \du^*), \nabla_h \dv)|\lesssim h^{-1}\|\Piu(\Piu^{k+2} u - \du)\|_0|\dv|_{1,h}.
$$
If $\ttp=\dph$, by the error estimates in \eqref{XGerr} and \eqref{err:vedge},
\begin{equation*}
\begin{aligned}
\left | \langle [\ttp],  \{\dv\} \rangle \right|
\le \eta^{-1/2}\|\eta^{1/2}[\dph]\|_\cE \|\{\dv\}\|_\cE\le h^{k+2}|\dv|_{1, h}(|\cp|_{k+2} + |u|_{k+1}),
\\
\left | \langle \{\cp - \ttp\},  [\dv] \rangle \right|
\le h^{-1/2}\|\cp - \dph\|_0\|\{\dv\}\|_\cE\le h^{k+2}|\dv|_{1, h}(|\cp|_{k+2} + |u|_{k+1}).
\end{aligned}
\end{equation*}
If $\ttp=\hp$, $\ttp\cdot \bn$ is continuous on interior edges. Thus, 
$$
\langle [\ttp],  \{\dv\} \rangle=0.
$$
The error estimates in \eqref{XGerr} and \eqref{err:vedge} imply that
\begin{equation*}
\begin{aligned} 
\left | \langle \{\cp - \ttp\},  [\dv] \rangle \right|
\le \left | \langle \{\cp - \dph\},  [\dv] \rangle \right|
+ \left | \langle \gamma [\dph] + \tp,  [\dv] \rangle \right|
\lesssim  h^{k+2}|\dv|_{1, h}(|\cp|_{k+2} + |u|_{k+1}).
\end{aligned}
\end{equation*}
Substituting the above estimates, Theorem \ref{th:super1} and 
$$
|(\Piu^{k+2} - I)u|_{1,h}\lesssim h^{k+2}|u|_{k+3}
$$
into \eqref{eq:vH1},
\begin{equation}\label{post2}
\begin{aligned} 
|\dv|_{1,h}\lesssim \|\alpha\nabla_h (\Piu^{k+2} - I)u\|_{0, \Om} + h^{k+2}+h^{-1}\|\Piu(\Piu^{k+2} u - \du)\|_0\lesssim h^{k+2}|u|_{k+3}.
\end{aligned}
\end{equation}
By the definition of $\Piu^{k}$, $\du^*$ and the superconvergence result in Theorem \ref{th:super1},
$$
\|\Piu(\Piu^{k+2} u - \du^*)\|_{0}=\|\Piu(\Piu^k u - \du)\|_{0}\lesssim h^{\min(2k+2, k+3)}(|\cp|_{k+2} + |u|_{k+1}).
$$
It follows \eqref{post1}, \eqref{post2} and the above estimate that 
\begin{equation*}
\begin{aligned} 
\|u-\du^*\|_0 \le \|u - \Piu^{k+2}u\|_0 + \|\Piu(\Piu^{k+2} u - \du^*)\|_{0} + \|\dv\|_0\lesssim h^{\min(2k+2, k+3)}|u|_{k+3},
\end{aligned}
\end{equation*} 
which completes the proof for $\du^*=u_{1, h}^*$ and $u_{2, h}^*$. The proof for $\du^*=u_{3, h}^*$ is  similar to the analysis in Theorem \ref{th:post2} for linear elasticity problem, which is omitted here.
\end{proof}
\begin{remark}
Similar to the analysis in \cite{hong2020extended} which is also presented in Section \ref{sec:elaintro} for the linear elasticity problem,  the formulation \eqref{XGdiv} with 
\begin{equation}\label{hybridchoice}
\tau=\mathcal{O}(h),\ \eta=\tau^{-1}, \ \gamma=0
\end{equation} 
 is hybridizable,  and can be reduced to a formulation with $\hp$. 
 By solving this reduced formulation  with much less computational cost, we can construct an approximation $u_{2, h}^*$ to $u$ with accuracy $\mathcal{O}(h^{\min(2k+2, k+3)})$ if the solution $u$ is smooth enough.
\end{remark}

\begin{remark}
For the first two postprocessing procedures, we let $u_{1, h}^*$ and $u_{2, h}^*$ in the discrete space $V_h^{k+2}$ to guarantee the superconvergence in Theorem \ref{th:post1}.
For a general mixed discontinuous Galerkin formulation with the conditions (C1)-(C3), if  $\du$ superconverges to the projection of the exact displacement, namely
$$
\|\Piu u - \du \|_{0, \Omega}\lesssim h^r \inf_{\dqh\in \Q_h, \dv\in V_h} (\|\cp - \dqh\|_{\rm div, h} + \|u-\dv\|_{0, h}),
$$
and $\displaystyle \|\cp - \dph\|_0\lesssim h^{\min (1,r-1)}\inf_{\dqh\in \Q_h, \dv\in V_h} (\|\cp - \dqh\|_{\rm div, h} + \|u-\dv\|_{0, h})$. 
We can choose a similar postprocessing technique by replacing $V_h^{k+2}$ in \eqref{postdef11} and \eqref{postdef12} by a large enough discrete space $\tilde V_h$ with $V_h\subset \tilde V_h$ and 
$$
\inf_{\dv\in \tilde V_h} \|u-\dv\|_0 + h|u-\dv|_{1,h}\lesssim h^{\min(2, r)} \inf_{\dqh\in \Q_h, \dv\in V_h} (\|\cp - \dqh\|_{\rm div, h} + \|u-\dv\|_{0, h}).
$$ 
Then, a similar analysis proves the superconvergence result
$$
\|u-u_h^*\|_0\lesssim h^{\min(2, r)} \inf_{\dqh\in \Q_h, \dv\in V_h} (\|\cp - \dqh\|_{\rm div, h} + \|u-\dv\|_{0, h}).
$$
\end{remark}

Next we introduce a Taylor expansion type postprocessing scheme, which follows \cite{bramble1989local}.
Recall that $\Piu^k$ is the $L^2$-projection onto $V_h^k$. Define the operator $\Piut^{k+2}$ onto $V_h^{k+2}$ by
\begin{equation*}
\left\{
\begin{split}
\int_K \partial^\alpha (u - \Piut^{k+2} u) \dx&=0,\quad \forall\  k+1\le |\alpha|\le k+2,
\\
\Piu^k (u - \Piut^{k+2} u)&=0.
\end{split}
\right.
\end{equation*}
Define $\phi_\alpha$ by $\phi_\alpha|_K={1\over \alpha!}(x-M_K)^\alpha$,
where $M_K$ is the centroid of element $K$. There exists the Taylor expansion
\begin{align}\label{eq:taylor}
(\Piut^{k+2} - \Piu^k)u = (I-\Piu^k)\Piut^{k+2}u = \sum_{|\alpha|=k+1}^{k+2} c_\alpha (I - \Piu^k )\phi_\alpha
\end{align}
with constants $c_\alpha$ to be determined.
Since
$$
\Piu^0\partial^\beta \phi_\alpha =\delta_{\alpha\beta},\quad
\partial^\beta  \Piut^{k+2}u= \sum_{|\alpha|=k+1}^{k+2} c_\alpha \partial^\beta\phi_\alpha
$$
for any $k+1\le |\alpha|, |\beta|\le k+2$, it holds that
\begin{align}
c_\alpha = \Piu^0\partial^\alpha  \Piut^{k+2} u = \Piu^0\partial^\alpha u,
\end{align}
which can be written as a function of $\cp=c\nabla u$, namely, $c_\alpha=c_\alpha(\cp)$.
Define the Taylor expansion type postprocessing $u_{4, h}^*\in V_h^{k+2}$ in \cite{bramble1989local}  by
\begin{align}\label{taylordef}
u_{4, h}^*=u_h + \sum_{|\alpha|=k+1}^{k+2} c_\alpha(\dph) (I - \Piu^k )\phi_\alpha.
\end{align}

The proof for the following theorem indicates that the same superconvergence result can be obtained if $\dph$ in \eqref{taylordef} is replaced by any high accuracy approximation to $\cp$.

\begin{theorem}\label{th:post1taylor}
Suppose $\cp\in \Hs^{k+2}(\Omega)$, $u\in H^{k+3}(\Omega) (k\ge 0)$, and  
$(\dph, \tp, \du, \tu)\in \Q_h^{k+1}\times  \checkQ_h^k\times V_h^k\times \check V_h^{k+1}$ is the solution of the four-field formulation \eqref{XGdiv} with $\eta=\left(\rho h_{e}\right)^{-1}, \tau \cong \eta^{-1}=\rho h_{e}$. It holds that
\begin{equation*} 
\|u- u_{4, h}^*\|_0 \lesssim h^{\min(2k+2, k+3)}|u|_{k+3}.
\end{equation*}  
\end{theorem}
\begin{proof}
Note that 
\begin{align*}
u - u_{4, h}^*&= (\Piu^k u - u_h ) + (u -  \Piut^{k+2}u)  +  (\Piut^{k+2}u -  \Piu^k u - \sum_{|\alpha|=k+1}^{k+2}  c_\alpha(\dph) (I - \Piu^k )\phi_\alpha)
\\
&= (\Piu^k u - u_h ) + (u -  \Piut^{k+2}u)  +  \sum_{|\alpha|=k+1}^{k+2}  (c_\alpha(\cp)  - c_\alpha(\dph) ) (I - \Piu^k )\phi_\alpha.
\end{align*}
By the definition of $c_\alpha(\cdot)$,
\begin{align}
\|c_\alpha(\cp)  - c_\alpha(\dph)\|_0\lesssim h^{-|\alpha|+1}\|\cp - \dph\|_0.
\end{align}
It follows from the above equation and the fact $\|(I - \Piu^k )\phi_\alpha\|_0\lesssim h^{|\alpha|}$ that
\begin{align}
\|u - u_{4, h}^*\|_0\le &\|\Piu^k u - u_h\|_0 + \|u -  \Piu^{k+2}u\|_0 +  \sum_{|\alpha|=k+1}^{k+2}  \|c_\alpha(\cp)  - c_\alpha(\dph)\|_0\|(I - \Piu^k )\phi_\alpha \|_0
\\
\le&\|\Piu^k u - u_h\|_0 + h^{k+3}|u|_{k+3} + h\|\cp - \dph\|_0.
\end{align}
A substitution of \eqref{err:L2} and Theorem \ref{th:super1} into the above inequality leads to
\begin{align}
\|u - u_{4, h}^*\|_0\lesssim h^{\min(2k+2, k+3)}(|\cp|_{k+2} + |u|_{k+1}),
\end{align}
which completes the proof.
\end{proof}

\section{Linear elasticity problems}\label{sec:elasticity}
This section analyzes the superconvergence result for the linear elasticity problem \eqref{XGdiv2}. 

\subsection{Discontinuous Galerkin method for linear elasticity problems}\label{sec:elaintro}
Consider the linear elasticity problem  \eqref{model:elasticity}. 
Let
\begin{align*}
\Qe_h:&=\{\dqh\in L^2(\Omega, \S): \dqh|_K\in \Q(K), \ \forall K\in \cTh_h\},\\
\V_h:&=\{v_h\in L^2(\Omega, \R^n): \dqh|_K\in V(K), \ \forall K\in \cTh_h\},\\
\checkQe_h:&=\{\tq\in L^2(\cE, \S): \dqh|_K\in \checkQ(K), \ \forall K\in \cTh_h\},\\
\checkV_h:&=\{\tv\in L^2(\cE, \R^n): \dqh|_K\in \check{V}(K), \ \forall K\in \cTh_h\},
\end{align*}
where $\Qe_h, \checkQe_{h}, \V_h, \checkV_{h}$  are subspaces of $L^2(\Omega, \S)$, $L^2(\cE, \S)$, $L^2(\Omega, \R^n)$ and $L^2(\cE, \R^n)$, respectively.
For any vector-valued function $\dve\in\V_h$ and tensor-valued function $\dtau\in\Qe_h$, let $\dve^{\pm}$ = $\dve|_{\partial K^{\pm}}$, $\dtau^{\pm}$ = $\dtau|_{\partial K^{\pm}}$.  Define the average $\{\cdot\}$ and the jump $[\cdot ]$ on interior edges/faces $e\in \cEi$ as
follows:  
\begin{equation}\label{jumpdef2}
\begin{array}{ll}
\{\dtau\}=\frac{1}{2}(\dtau^++\dtau^-),&[\dtau]=\dtau^+\bn^++\dtau^-\bn^-,\\
\{\dve\}=\frac{1}{2}(\dve^++\dve^-),&[\dve] =\dve^+\odot \bn^++\dve^-\odot \bn^- - (\dve^+\cdot \bn^+ + \dve^-\cdot \bn^-)I,
\end{array}
\end{equation} 
where $\dve\odot \bn= \dve\bn^T+ \bn\dve^T$. 
For any boundary edge/face $e\subset \partial \Omega$, define
\begin{equation}\label{bddef2}
\begin{array}{lllll}
\{\dtau\}=\dtau, &  [\dtau]=0, &\{\dve\}=\dve,&[\dve]=\dve\odot \bn -(\dve\cdot \bn)I,& \text{on }\Gamma_D,\\
\{\dtau\}=\dtau,&  [\dtau]=\dtau\bn, &\{\dve\}=\dve, & [\dve]=0,& \text{on }\Gamma_N.
\end{array}
\end{equation}   
With the aforementioned definitions,  the following identities \cite{arnold2002unified} holds:
\begin{equation}\label{identities}
\langle \dtau\bn, \dve\rangle_\PcT = \langle \{\dtau\}, [\dve]\rangle + \langle [\dtau], \{\dve\}\rangle,\quad \forall \dtau\in \Qe_h,\ \dve\in \V_h.
\end{equation}
For any vector-valued function $\dve$ and tensor-valued function $\dtau$,  define the piecewise symmetric strain tensor $\epsilon_h$ and piecewise divergence $\divh$ by
$$
\epsilon_h (\dve)\big |_K=\epsilon (\dve|_K), \quad
\divh \dtau\big |_K=\mbox{div} (\dtau |_K), \quad \forall K \in \cTh_h.
$$ 
There exists a similar DG identity  to \eqref{DGidentity} as below
\begin{equation}\label{DGidentity2}
(\dtau, \epsilon_h(\dve))
=-(\divh \dtau, \dve)
+ \langle [\dtau], \{\dve\}\rangle
+ \langle \{\dtau\}\bn, [\dve]\bn\rangle,\quad
\mbox{if}\  \dtau\in \Qe_h.
\end{equation}

For the linear elasticity problem \eqref{model:elasticity}, consider the four-field extended Galerkin formulation in \cite{hong2020extended}, which seeks  $(\dsig, \hsig, \due, \hue)\in \Qe_h\times  \checkQe_h\times \V_h\times \checkV_h$ such that 
\begin{equation}\label{XGdiv2}
\left\{
\begin{array}{rll}
(A\dsig,\dtau) 
+(\due, \divh \dtau) 
-\langle \{\due\} + \hue - (\gamma\cdot \bn)[\due]\bn, [\dtau]\rangle
&=0,&\ \forall \dtau \in \Qe_h,
\\
(\divh \dsig, \dve) 
-\langle [\dsig], \{\dve\}\rangle 
+\langle \hsig +  [\dsig]\gamma^T, [\dve]\rangle
&=(f,\dve)  &\ \forall \dve\in \V_h,
\\
\langle \tau^{-1}   
\hsig+ [\due], \htau \rangle_{e }
&=0,&\forall  \htau \in \checkQe_{h},
\\
\langle \eta^{-1} \hue + [\dsig], \hve \rangle_e
&=0,&\forall  \hve \in \checkV_{h}.
\end{array}
\right.
\end{equation}  
For any $(\dtau, \htau, \dve, \hve)\in \Qe_h\times \checkQe_{h}\times \V_h\times \checkV_{h}$, define
\begin{align*}
\|\dtau\|_{\rm div, h}^2 &= (A\dtau, \dtau) + \|\divh \dtau\|_{0}^2 + \|\eta^{1/2}[\dtau]\|_\cE^2, & \|\htau\|_{0, h}^2&=\|\tau^{-1/2}\htau\|_\cE^2,
\\
\|\dve\|_{0, h}^2 &= \|\dve\|_{0}^2 + \|\tau^{1/2}[\dve]\|_\cE^2,& \|\hve\|_{0, h}^2&=\|\eta^{-1/2}\hve\|_\cE^2,
\end{align*}
and the $L^2$ norm of $\dtau$ by
$$
\|\dtau\|_A^2 = (A\dtau, \dtau),\quad \forall \dtau \in L^2(\Omega, \S).
$$

For $H({\rm div})$-based formulations \eqref{XGdiv2}, the well-posedness and the error estimate are analyzed in \cite{hong2020extended} under a set of assumptions as listed in the following lemma. 
\begin{lemma}\label{lm:elasticity}
The four-field formulation \eqref{XGdiv2} which satisfies the conditions
\begin{enumerate}
\item[(A1)] $\Qe_h=\Qe_{h}^{k+1}$, $\divh \Qe_h=\V_h\subset \V_h^k$, $k\ge 0$;
\item[(A2)] $\checkV_h^{k+1}\subset \checkV_h$;
\item[(A3)] $\tau=\rho_1 h_e $, $\eta=\rho_2^{-1} h_e^{-1}$ and there exist positive constants $C_1$, $C_2$,  $C_3$ and $C_4$ such that
$$
0\leq \rho_1\leq C_1,\quad C_2\leq \rho_2\leq C_3,\quad 0\leq \gamma\leq C_4,
$$  
\end{enumerate}
 is uniformly well-posed with respect to the norms when $\rho_1$ and $\rho_2$. Namely, if $\csig\in H^{k+2}(\Om, \S)$, $u\in H^{k+1}(\Om, \R^n) ( k\ge0 )$ and let $(\dsig, \hsig, \due, \hue)\in \Qe_h^{k+1}\times \checkQe_{h}^k\times \V_h^{k}\times \checkV_{h}^{k+1}$ be the solution of \eqref{XGdiv2}, then we have the following error estimate:
\begin{equation}\label{XGerr2}
\|\csig-\dsig\|_{\rm div, h}+\|\hsig\|_{0, h} + \|\cu-\due\|_{0, h} +\|\hue\|_{0, h}\lesssim h^{k+1}(|\csig|_{k+2} + |\cu|_{k+1}).
\end{equation} 
Furthermore, if $k\ge n$, it holds that
\begin{equation}\label{err:L22}
\| \csig - \dsig\|_A\lesssim h^{k+2}(|\csig|_{k+2} + |\cu|_{k+1}).
\end{equation}
\end{lemma}
Here discrete spaces $\Qe_h^{k}$, $\checkQe_h^k$, $\V_h^k$ and $\checkV_h^{k}$ are subspaces of $L^2(\Omega, \S)$,  $L^2(\cE, \S)$, $L^2(\Omega, \R^n)$ and $L^2(\cE, \R^n)$, respectively, and contain all piecewise polynomials of degree not larger than $k$.

The analysis in \cite{hong2020extended} shows that a special case of \eqref{XGdiv2} is hybridizable as presented below. 
Denote
$$
\Z_h = \{\due\in \V_h: \epsilon_h (\due)=0\},
$$
$$
\V_h^\perp = \{\due\in \V_h: (\due, \dve)=0,\ \forall \dve\in \Z_h\}.
$$
\begin{theorem}\label{th:hybride}
The formulation \eqref{XGdiv2}  with discrete spaces satisfying the assumptions in Lemma \ref{lm:elasticity} and condition \eqref{hybridchoice} 
can be decomposed into two sub-problems as:
\begin{enumerate}
\item[(I)] Local problems. For each element $K$, given $\tsig \in  \checkQe_{h}$, find $(\dsig^K, \due^K)\in \Qe_h\times \V_h^\perp$ such that for any $(\dtau, \dve)\in \Qe_h\times \V_h^\perp$
\begin{equation}\label{XGWG1e}
\left\{
\begin{array}{rll}
(A\dsig^K,\dtau)_K
- (\epsilon_h (\due^K), \dqh)_K
+\langle  \eta\dsig^K \bn, \dtau\bn\rangle_{\partial K}
&
=\langle  \eta\tsig\bn, \dtau\bn\rangle_{\partial K},
\\
-(\dsig^K, \epsilon_h (\dve))_K
&
=(f,\dve)_K-\langle   \tsig\bn, \dve\rangle_{\partial K}.
\end{array}
\right.
\end{equation}
Denote $W_{\Qe}: \checkQe_h\rightarrow \Qe_h$ and $W_{\V}: \checkQe_h\rightarrow \V_h^\perp$ by 
$$ 
W_{\Qe}(\tsig)|_K= \dsig^K\quad\text{ and }\quad  W_{\V}(\tsig)|_K= \due^K,
$$
respectively. 
\item[(II)] Global problem. Find $(\tsig, \due^0)\in \checkQe_{h}\times Z_h$ such that  for any $\dve^0\in Z_h$ and $\htau \in \checkQe_{h}$,
{\small
\begin{equation} \label{XGWG2e}
\left\{
\begin{array}{rl}
\langle \eta (\tsig - W_Q(\tsig))\bn, (\htau-W_Q(\htau))\bn \rangle_\PcT
+\langle \due^0, W_V(\htau) \rangle_\PcT
&= - (f, W_{\V}(\htau)),
\\
\langle   \hsig\bn, \dve^0\rangle_\PcT,
&
=(f,\dve^0).
\end{array}
\right.
\end{equation}}
\end{enumerate}
Let $(\tsig, \due^0)$ be the solution of \eqref{XGWG2e}, $(\dsig^K, \due^K)$ be the solution of \eqref{XGWG1e}, and $(\dsig, \hsig, \due, \hue)$ be the solution of \eqref{XGdiv2}. Then,
$$
\dsig^K=\dsig|_K, \ \due^K + \due^0= \due|_K, \ \tsig = \hsig + \{\dsig\}.
$$
\end{theorem}
Theorem \ref{th:hybride} indicates that the discontinuous Galerkin formulation \eqref{XGdiv2} with this special choice \eqref{hybridchoice} of parameters can be written as a system of $\tsig$ and $\due^0$, which reduces the degree of freedom and the computational cost.

\subsection{Superclose analysis for linear elasticity problems}
This section considers the superclose result for linear elasticity problems \eqref{XGdiv2}. The  analysis for the superclose property requires two main ingredients: a conforming interpolation onto 
$\Qe_h$, and the commuting property of this interpolation.

Let $\Piue$ be  the standard $L^2$-projection onto $\V_h$, namely
$$
(\Piue \cu, \dve)=(\cu, \dve),\quad \forall \dve\in \V_h.
$$
For $\V_h=\V_h^k$, denote the $L^2$-projection by $\Piue^k$. The analysis for the linear elasticity problem requires the following assumption
\begin{Ass}\label{ass:commuting}
There exists a projection $\Pipe$ onto a conforming subspace $\Qe_h^c$ of $\Qe_h$, and   the projection $\Pipe$ admits the commuting diagram
\begin{equation}\label{pro:commuting}
{\rm div} \Pipe \ctau = \Piue {\rm div} \ctau,\quad \forall \ctau\in \Hs({\rm div}, \Omega).
\end{equation} 
\end{Ass}

Let $(\dsig, \hsig, \due, \hue)\in \Qe_h\times  \checkQe_h\times \V_h\times \checkV_h$ be the solution of the four-field formulation \eqref{XGdiv2}.
Define  
\begin{equation}\label{errdef}
\epu =\Piue \cu - \due, \quad \ep =\csig-\dsig.
\end{equation}

\begin{lemma}  
Suppose that the conditions (A1)-(A3) and the Assumption \ref{ass:commuting} hold. For any $\psie\in L^2(\Omega, \R^n)$, let $\phie$ be the solution of Problem \eqref{model:elasticity} with $f=\psie$, which implies that ${\rm div}(A^{-1}\epsilon(\phie))=\psie$. It holds that  
\begin{equation}\label{eq:decom}
\begin{aligned}
(\epu, \psie)  =& (\divh \ep, (I - \Piue)\phie) + (A\ep, (I - \Pipe)(A^{-1}\epsilon(\phie))) 
- \langle [\dsig], \{(\Piue - I)\phie\}\rangle
\\
& +\langle \hsig + [\dsig]\gamma^T, [\Piue\phie]\rangle. 
\end{aligned}
\end{equation}
\end{lemma}
\begin{proof}
Note that the formulation \eqref{XGdiv2} is consistant, namely, $(\csig, 0, \cu, 0)$ satisfies \eqref{XGdiv2}. 
Let
\begin{equation}\label{hatdef2}
\tsig =  \{\dsig\} + [\dsig]\gamma^T+\hsig, \qquad
\tue = \{\due\} - (\gamma\cdot \bn)[\due]\bn + \hue
\end{equation}
with $\gamma\in \mathbb{R}^{n\times 1}$. 
By the DG identity \eqref{DGidentity2}, the formulation \eqref{XGdiv2} and $\divh \Qe_h\subset \V_h$,
\begin{equation}\label{erreqdiv}
\left\{
\begin{aligned}
(A\ep, \dtau) + (\epu, \divh \dtau) &= \langle \{\cu-\tue\}, [\dtau]\rangle &\forall \dtau\in \Qe_h
\\
-(\divh \ep, \dve) &= \langle \{\dsig - \tsig\}, [\dve]\rangle + \langle [\dsig], \{\dve\}\rangle &\forall \dve\in \V_h.
\end{aligned}
\right.
\end{equation}
For any $\psie\in \V_h$, since $\epu\in \V_h$, by the commuting diagram \eqref{pro:commuting},
\begin{equation}
\begin{aligned}
(\epu, \psie)&= (\epu, {\rm div} \Pipe (A^{-1}\epsilon(\phie))).
\end{aligned}
\end{equation} 
Let  $\dtau= \Pipe (A^{-1}\epsilon(\phie))$ in \eqref{erreqdiv}. It follows from $[\dtau]=0$ that
\begin{equation}\label{eq:1}
\begin{aligned}
(\epu, \psie)&= -(A\ep, \Pipe (A^{-1}\epsilon(\phie))) 
=  -(\ep, \epsilon(\phie)) + (A\ep, (I - \Pipe)(A^{-1}\epsilon(\phie))). 
\end{aligned}
\end{equation} 
Let $\dve=\Piue\phie$ in \eqref{erreqdiv}. It holds that
\begin{equation}\label{eq:2}
\begin{aligned}
(\divh \ep, \phie) =& (\divh \ep, (I - \Piue)\phie)  
- \langle \{\dsig - \tsig\}, [\Piue\phie]\rangle - \langle [\dsig], \{\Piue\phie\}\rangle.
\end{aligned}
\end{equation} 
A combination of \eqref{eq:1}, \eqref{eq:2} and
\begin{equation}
\begin{aligned}
(\ep, \epsilon(\phie)) &= -(\divh \ep, \phie) + \langle [\ep], \{\phie\}\rangle
\end{aligned}
\end{equation} 
gives
\begin{equation}\label{eq:3}
\begin{aligned}
(\epu, \psie)  =& (\divh \ep, (I - \Piue)\phie) + (A\ep, (I - \Pipe)(A^{-1}\epsilon(\phie)) )  
- \langle [\dsig], \{\Piue\phie\}\rangle  
\\
&
- \langle [\ep], \{\phie\}\rangle
- \langle \{\dsig - \tsig\}, [\Piue\phie]\rangle.
\end{aligned}
\end{equation} 
According to \eqref{errdef} and \eqref{hatdef2},
\begin{equation}\label{eq:4}
\begin{aligned}
\langle [\dsig], \{\Piue\phie\}\rangle + \langle [\ep], \{\phie\}\rangle = &\langle [\dsig], \{(\Piue - I)\phie\}\rangle,
\\
\langle \{\dsig - \tsig\}, [\Piue\phie]\rangle =& -\langle \hsig + [\dsig]\gamma^T, [\Piue\phie]\rangle.
\end{aligned}
\end{equation} 
Substituting \eqref{eq:4} into \eqref{eq:3},
\begin{equation*} 
\begin{aligned}
(\epu, \psie)  =& (\divh \ep, (I - \Piue)\phie) + (A\ep, (I - \Pipe)(A^{-1}\epsilon(\phie))) 
- \langle [\dsig], \{(\Piue - I)\phie\}\rangle
\\
& +\langle \hsig + [\dsig]\gamma^T, [\Piue\phie]\rangle,
\end{aligned}
\end{equation*} 
which completes the proof.
\end{proof}

It was analyzed in \cite{hu2015finite} that there exists such a conforming interpolation $\Pipe$ with commuting property \eqref{pro:commuting} for $k\ge n$, and it holds that
\begin{equation}\label{picest}
\| \csig - \Pipe \csig\|_0\lesssim h^{k+2}|\csig|_{k+2}.
\end{equation}
The following theorem shows that $\|\epu\|_{0}$ converges at the rate $k+3$ if solutions are smooth enough. The accuracy is presented in the form of $h^{\min(2k+2, k+3)}$ to be consistent with the result in Theorem \ref{th:super1} for scalar elliptic problems.
\begin{theorem}\label{th:super2}
Suppose $\csig\in H^{k+2}(\Omega, \S)$, $\cu\in H^{k+1}(\Omega, \R^n) (k\ge n)$, and  
$(\dsig, \hsig, \due, \hue)$, which is in $ \Qe_h^{k+1}\times  \checkQe_h^k\times \V_h^k\times \checkV_h^{k+1}$, is the solution of the four-field formulation \eqref{XGdiv2} with $\eta=\left(\rho h_{e}\right)^{-1}, \tau \cong \eta^{-1}=\rho h_{e}$. It holds that
\begin{equation}\label{superclose}
\|\Piue^k \cu - \due\|_{0, \Omega}\lesssim h^{\min(2k+2, k+3)}(|\csig|_{k+2} + |\cu|_{k+1}).
\end{equation}
\end{theorem}
\begin{proof} 
Since $\Piue^k$ is the $L^2$-projection onto $\V_h^{k}$,
\begin{equation}\label{eq:errorPi}
\|(I - \Piue^k)\cv\|_{0, K}\lesssim h^{k+1}|\cv|_{k+1, K},\quad \forall \cv\in H^{k+1}(K, \R^n).
\end{equation}
By the triangle inequality, \eqref{XGerr2} and \eqref{eq:errorPi},
\begin{equation}\label{term2}
\left |(\divh \ep, (I - \Piue^k)\phie) \right |\le \|\divh \ep\|_{0}\|(I - \Piue^k)\phie\|_{0}\lesssim h^{\min(2k+2, k+3)}|\phie|_2(|\csig|_{k+2} + |\cu|_{k+1}).
\end{equation}  
The $L^2$ error estimate of $\|\ep\|_{0}$ in \eqref{err:L22}, \eqref{picest} and \eqref{eq:errorPi} indicate that
\begin{equation}\label{term3}
\begin{aligned}
\left |  (A\ep, (I - \Pipe)(A^{-1}\epsilon(\phie)) \right |&\le \|A\ep\|_{0} \|(I - \Pipe)(A^{-1}\epsilon(\phie))\|_{0}
\\
&\lesssim h^{\min(2k+4, k+3)}|\phie|_2(|\csig|_{k+2} + |\cu|_{k+1}).
\end{aligned}
\end{equation}  
It follows from the error estimates \eqref{XGerr2}, \eqref{eq:errorPi} and trace inequality that
\begin{equation}\label{term4}
\begin{aligned}
\left |\langle [\dsig], \{(I - \Piue^k)\phie\}\rangle \right |&\le (\eta h)^{-1/2}\|\eta^{1/2}[\dsig]\|_{\cE}\|(I - \Piue^k)\phie\|_{0}
\\
&\lesssim h^{\min(2k+2, k+3)}|\phie|_2(|\csig|_{k+2} + |\cu|_{k+1}),
\end{aligned}
\end{equation}
\begin{equation}\label{term5}
\begin{aligned}
\left | \langle \hsig + [\dsig]\gamma^T, [\Piue^k\phie]\rangle \right |&\le h^{-{1\over 2}}\|\hsig + [\dsig]\gamma^T\|_\cE\|(I - \Piue^k)\phie\|_0
\\
&\lesssim  h^{\min(2k+2, k+3)}|\phie|_2(|\csig|_{k+2} + |\cu|_{k+1}).
\end{aligned}
\end{equation} 
A substitution of \eqref{term2}, \eqref{term3}, \eqref{term4} and \eqref{term5} 
into \eqref{eq:decom} leads to  
\begin{equation}
\left |(\epu, \psie)\right |\lesssim h^{\min(2k+2, k+3)}|\phie|_2(|\csig|_{k+2} + |\cu|_{k+1}).
\end{equation}
Since $|\phie|_2\lesssim \|\psie\|_0$,
\begin{equation}
\|\epu\|_0=\sup_{0\neq \psie\in L^2(\Omega)}{(\epu, \psie)\over \|\psie\|_0}\lesssim h^{\min(2k+2, k+3)}(|\csig|_{k+2} + |\cu|_{k+1}),
\end{equation}
which completes the proof.
\end{proof}

\begin{remark}
Since the four-field extended Galerkin method recovers most of discontinuous Galerkin methods in literature \cite{XGHong,hong2020extended}, Theorem \ref{th:super1} and Theorem \ref{th:super2} imply that most of the $H({\rm div})$-based discontinuous Galerkin methods in literature \cite{cockburn2009unified,hong2019unified,MixLDGWu,hong2020extended} admit this superclose property.
\end{remark}



\subsection{Postprocess technique for linear elasticity problems}\label{sec:postela}

Consider the linear elasticity problems \eqref{model:elasticity}.
Denote the rigid motion, the kernel of the symmetric gradient operator $\epsilon(\cdot)$, by
$$
{\rm RM}(K, \R^2)=\mbox{span}\left\{\begin{pmatrix} 1\\0\end{pmatrix},
\begin{pmatrix} 0\\1\end{pmatrix},
\begin{pmatrix} y\\-x\end{pmatrix}\right\}.
$$
For any $\cv\in L^2(K, \R^2)$, define $L^2$-projection onto ${\rm RM}(K, \R^2)$ by $ \Piue^*v$, namely,
$$
\int_K (I - \Piue^*)\cv\cdot \dw dx=0,\quad \forall\ \dw\in {\rm RM}(K, \R^2).
$$
Note that for any positive integer $k\ge 1$,
\begin{equation}\label{projectast}
\Piue^* \Piue^k\cv=\Piue^*\cv.
\end{equation}

Consider the $H({\rm div})$-based four-field formulation \eqref{XGdiv2} with $\Qe_h=\Qe_h^{k+1}$, $\checkQe_h=\checkQe_h^k$, $\V_h= \V_h^k$ and $\checkV_h= \checkV_h^{k+1}$. Lemma \ref{lm:elasticity} guarantees the wellposedness of this problem.
We introduce a new postprocess procedure for linear elasticity problem. Let $\due^*\in \V_h^{k+2}$ be the solution of the following problem
\begin{equation}\label{post:elasticity}
\left\{
\begin{aligned}
(\epsilon (\due^*), \epsilon(\dve))_K&=(A\dsig, \epsilon(\dve))_K,& \forall \dve\in P_{k+2}(K, \R^2)\\
\Piue^*(\due^* - \due)&=0,&
\end{aligned}
\right.
\end{equation} 
where  $(\dsig, \hsig, \due, \hue)$ is the solution of the mixed discontinuous Galerkin formulation  \eqref{XGdiv2}.  

The following theorem illustrates that the postprocessing solution $\due^*$ admits a higher accuracy compared to the approximation $\due$.
\begin{theorem}\label{th:post2}
Suppose $\csig\in \Hs^{k+2}(\Omega)$, $\cu\in H^{k+3}(\Omega) (k\ge n)$, and  
$(\dsig, \hsig, \due, \hue)\in \Qe_h^{k+1}\times \checkQe_h^k\times \V_h^k\times \checkV_h^{k+1}$ is the solution of the four-field formulation \eqref{XGdiv2} with $\eta=\left(\rho h_{e}\right)^{-1}, \tau \cong \eta^{-1}=\rho h_{e}$. It holds that
\begin{equation*} 
\|\cu-\due^*\|_0 \lesssim h^{\min(2k+2, k+3)}|\cu|_{k+3}.
\end{equation*} 

\end{theorem}
\begin{proof}
A combination of \eqref{model:elasticity} and \eqref{post:elasticity} gives
\begin{align}\label{eq:epsilon1}
(\epsilon (\due^*) - \epsilon (\Piue^{k+2}\cu), \epsilon(\dve)) &=(A(\dsig - \csig), \epsilon(\dve)) 
+ (\epsilon (\cu) - \epsilon (\Piue^{k+2}\cu), \epsilon(\dve)) 
\end{align}
According to Lemma \ref{lm:elasticity} and the definition of $\Piue^{k+2}$,
$$
\|\dsig - \csig\|_A\lesssim h^{k+2}(|\csig|_{k+2} + |\cu|_{k+1}),\qquad \|\epsilon (\cu) - \epsilon (\Piue^{k+2}\cu)\|_0\lesssim h^{k+2}|\cu|_{k+3}.
$$
Let $\dve = \due^* - \Piue^{k+2}\cu$ in \eqref{eq:epsilon1}. It follows from \eqref{err:L22} that
\begin{equation}\label{eq:epsilon2}
\|\epsilon (\dve)\|_0\le \|A(\dsig - \csig)\|_{0} + \|\epsilon (\cu) - \epsilon (P_h^{k+2}\cu)\|_0\lesssim h^{k+2}|\cu|_{k+3}.
\end{equation}
Since $k\ge n\ge 1$, by \eqref{projectast} and Theorem \ref{th:super2},
\begin{equation}\label{eq:posth1}
\|\Piue^*\dve\|_0=\|\Piue^*(\due - \Piue^{k}\cu)\|_{0}\le \|\due - \Piue^{k}\cu\|_0\lesssim h^{k+3}|\cu|_{k+3}.
\end{equation} 
Since $(I - \Piue^*)\dw=0$ for any $\dw\in {\rm RM}(K, \R^2)$, it follows from \eqref{eq:epsilon2} and the scaling technique that 
\begin{equation}\label{eq:posth2}
\|(I - \Piue^*)\dve\|_{0}\lesssim h\|\epsilon(\dve)\|_{0}\lesssim h^{k+3}|\cu|_{k+3}.
\end{equation} 
A combination of \eqref{eq:posth1} and \eqref{eq:posth2} gives 
$$
\|\dve\|_0\le \|\Piue^*\dve\|_{0} + \|(I - \Piue^*)\dve\|_{0} \lesssim h^{k+3}|\cu|_{k+3}.
$$
Consequently,
$$
\|\due^*- \cu\|_{0}\le \|\dve\|_{0} + \|\Piue^{k+2}\cu - \cu\|_{0}\lesssim h^{k+3}|\cu|_{k+3},
$$
which completes the proof.
\end{proof}
\begin{remark}
For the case $k<n$, the analysis in \cite{hong2020extended} indicates that as long as  Assumption \ref{ass:commuting} holds with 
\begin{equation}
\| \ctau - \Pipe \ctau\|_0 \lesssim h^{k+2}|\ctau|_{k+2}, \ \forall \ctau \in H^{k+2}(\Omega, \S),
\end{equation}
there exists
$$
\|\csig - \dsig\|_0\lesssim h^{k+2} (|\csig|_{k+2} + |\cu|_{k+1}),
$$
where $(\dsig, \hsig, \due, \hue)$ is the solution of the discontinuous Galerkin formulation \eqref{XGdiv2} in $ \Qe_h^{k+1}\times  \checkQe_h^k\times \V_h^k\times \checkV_h^{k+1}$. This means that Assumption \ref{ass:commuting} guarantees the superclose property \eqref{superclose}, which implies  the superconvergence of the postprocessed approximation $\due^*$ following the analysis of Theorem \ref{th:post2}. The numerical results in Table \ref{tab:elasticity3} and  \ref{tab:elasticity4} show that $\due^*$ converges at the same rate as $\due$ for $k<n$. This implies that Assumption \ref{ass:commuting} is not true for $k<n$, namely there exists no such $H({\rm div})$-conforming projection for low order discrete spaces.
\end{remark}

\begin{remark} 
For a general mixed discontinuous Galerkin formulation \eqref{XGdiv2} with the conditions (A1)-(A3), if there holds
$$
\|\csig - \dsig \|_{0, \Omega}\lesssim h^r \inf_{\dqh\in \Qe_h, \dve\in \V_h} (\|\csig - \dtau\|_{\rm div, h} + \|\cu-\dve\|_{0, h}),
$$
and 
$$
\|(I-\Piue)\phie\|_0\lesssim h^t|\phie|_2,\quad \|(I-\Pipe )(A^{-1}\epsilon(\phie))\|_0\lesssim h^s|\phie|_2.
$$ 
Then, a similar analysis proves the superconvergence result
$$
\|\cu-\due^*\|_0\lesssim h^{\min(s+r, t)} \inf_{\dtau\in \Qe_h, \dve\in \V_h} (\|\csig - \dtau\|_{\rm div, h} + \|\cu-\dve\|_{0, h}).
$$
\end{remark}

\section{Numerical Tests}\label{sec:num}

In this section, some numerical experiments in 2D are presented to verify the estimate provided in Theorem \ref{th:super1}, \ref{th:post1}, \ref{th:super2}  and \ref{th:post2}. 

\subsection{Example 1: scalar elliptic problems}
We consider the model problem \eqref{model:elliptic} on the unit square $\Om=(0,1)^2$ with 
$$
u= \sin (\pi x)\sin(\pi y),
$$
and set $f$ and $g$ to satisfy the above exact solution of \eqref{model:elliptic}.
The domain is partitioned by uniform triangles. The level one triangulation $\mathcal{T}_1$ consists of two right triangles, obtained by cutting the unit square with a north-east line. Each triangulation $\mathcal{T}_i$ is refined into a half-sized triangulation uniformly, to get a higher level triangulation $\mathcal{T}_{i+1}$.  

Consider the four-field formulation \eqref{XGdiv} with $\eta=h_e^{-1}$, $\tau=h_e$, $\gamma=1$ and 
$$
\Q_h=\Q_h^{\alpha_1},\ \checkQ_h=\checkQ_h^{\alpha_2},\ V_h = V_h^{\alpha_3}, \ \check V_h= \check V_h^{\alpha_4},
$$
where $\alpha=(\alpha_1, \alpha_2, \alpha_3, \alpha_4)$ satisfies $\alpha_1=\alpha_4=k+1$, $\alpha_2=\alpha_3=k$ for $k=0$, $1$ and $2$. According to Lemma \ref{lm:elliptic}, these formulations are well posed. Denote the corresponding solution by $(\dph, \tp, \du, \tu)$.  

\begin{table}[htbp] 
\small
  \centering
    \begin{tabular}{|c|c|c|c|c|c|c|c|c|}
\hline
 & $\|u - \du\|_{0}$  &  rates     &$\|\du - \Piu u\|_{0}$ &   rates      & $\|u - u_{1, h}^*\|_{0}$  &   rates     &$\|u - u_{2, h}^*\|_{0}$    &  rates   \\\hline
$\cTh_3$ & 1.45E-01 & 0.92  & 6.81E-02 & 0.93  & 6.91E-02 & 1.00  &  6.87E-02 & 1.04 \\\hline
$\cTh_4$ & 6.89E-02 & 1.08  & 2.26E-02 & 1.59  & 2.27E-02 & 1.61  &  2.26E-02 & 1.60  \\\hline
$\cTh_5$ & 3.33E-02 & 1.05  & 6.24E-03 & 1.86  & 6.25E-03 & 1.86  &   6.24E-03 & 1.86 \\\hline
$\cTh_6$ & 1.64E-02 & 1.02  & 1.62E-03 & 1.95  & 1.62E-03 & 1.95  &  1.62E-03 & 1.95 \\\hline
$\cTh_7$ & 8.19E-03 & 1.00  & 4.11E-04 & 1.98  & 4.11E-04 & 1.98  & 4.11E-04 & 1.98 \\\hline
$\cTh_8$ & 4.09E-03 & 1.00  & 1.03E-04 & 1.99  & 1.03E-04 & 1.99  &  1.03E-04 & 1.99  \\\hline
    \end{tabular}%
  \caption{Superconvergence for the scalar elliptic problem using $\alpha=(1,0,0,1)$.}
  \label{tab:elliptic0}%
\end{table}%

\begin{table}[htbp]
\small
  \centering 
    \begin{tabular}{|c|c|c|c|c|c|c|c|c|}
\hline
  & $\|u - \du\|_{0}$  &  rates     &$\|\du - \Piu u\|_{0}$&   rates      & $\|u - u_{1, h}^*\|_{0}$  &   rates     &$\|u - u_{2, h}^*\|_{0}$    &  rates\\\hline
$\cTh_3$ & 1.96E-02 & 1.95  & 1.99E-03 & 3.35  & 2.41E-03 & 3.42  &  2.22E-03 & 3.52\\\hline
$\cTh_4$ & 4.95E-03 & 1.98  & 1.43E-04 & 3.80  & 1.68E-04 & 3.84  & 1.49E-04 & 3.89 \\\hline
$\cTh_5$ & 1.24E-03 & 1.99  & 9.40E-06 & 3.92  & 1.10E-05 & 3.94  & 9.63E-06 & 3.96 \\\hline
$\cTh_6$ & 3.11E-04 & 2.00  & 6.02E-07 & 3.97  & 7.00E-07 & 3.97  & 6.11E-07 & 3.98  \\\hline
$\cTh_7$ & 7.78E-05 & 2.00  & 3.80E-08 & 3.98  & 4.42E-08 & 3.99  & 3.85E-08 & 3.99 \\\hline
    \end{tabular}%
  \caption{Superconvergence for the scalar elliptic problem using $\alpha=(2,1,1,2)$.}
  \label{tab:elliptic1}%
\end{table} 

\begin{table}[htbp] 
\small
  \centering 
    \begin{tabular}{|c|c|c|c|c|c|c|c|c|}
\hline
 & $\|u - \du\|_{0}$  &  rates     &$\|\du - \Piu u\|_{0}$  &   rates      & $\|u - u_{1, h}^*\|_{0}$  &   rates     &$\|u - u_{2, h}^*\|_{0}$    &  rates \\\hline
$\cTh_3$ & 2.17E-03 & 2.92  & 8.89E-05 & 4.42  & 3.48E-04 & 3.91  & 1.55E-04 & 3.78 \\\hline
$\cTh_4$ & 2.75E-04 & 2.98  & 3.20E-06 & 4.80  & 1.34E-05 & 4.70  &  6.07E-06 & 4.68  \\\hline
$\cTh_5$ & 3.45E-05 & 2.99  & 1.06E-07 & 4.92  & 4.50E-07 & 4.90  &  2.04E-07 & 4.89 \\\hline
$\cTh_6$ & 4.31E-06 & 3.00  & 3.39E-09 & 4.97  & 1.44E-08 & 4.96  & 6.56E-09 & 4.96  \\\hline
$\cTh_7$ & 5.39E-07 & 3.00  & 1.07E-10 & 4.98  & 4.57E-10 & 4.98  & 2.08E-10 & 4.98 \\\hline
    \end{tabular}%
  \caption{Superconvergence for the scalar elliptic problem using $\alpha=(3,2,2,3)$.}
  \label{tab:elliptic2}%
\end{table}%

Table \ref{tab:elliptic0} - \ref{tab:elliptic2} record the errors $\|u - \du\|_{0}$, $\|\Piu u - \du\|_{0}$,  $\|u - u_{1, h}^*\|_{0}$, $\|u - u_{2, h}^*\|_{0}$  and the corresponding convergence rates for the aforementioned four-field formulations \eqref{XGdiv} with $\Piu=\Piu^0$ in \eqref{postdef11} and \eqref{postdef12}. It reveals in these tables that $\|\Piu u - \du\|_{0}$,  $\|u - u_{1, h}^*\|_{0}$ and $\|u - u_{2, h}^*\|_{0}$ converge at 
the same rate 2.00 if $\alpha=(1,0,0,1)$, 4.00 if $\alpha=(2,1,1,2)$ and 5.00 if $\alpha=(3,2,2,3)$,
 which coincides with the analysis in Theorem \ref{th:super1} and Theorem \ref{th:post1}. The comparison between $\|u - u_{1, h}^*\|_{0}$ and $\|u - u_{2, h}^*\|_{0}$ shows that the postprocessing approximations $u_{2, h}^*$ admit a slightly higher accuracy than $u_{1, h}^*$.  It is analyzed in \cite{XGHong} that the four-field formulation \eqref{XGdiv} with 
 $\eta= \tau^{-1}$ 
 and $\gamma=0$ is hybridizable. For this formulation, the postprocess technique \eqref{postdef12} with $\ttp=\hp$ is better than the other one \eqref{postdef11} in two aspects, one is  the higher accuracy of $u_{2, h}^*$ and the other one is that there is no need to solve $\dph$ from the reduced  formulation.

\subsection{Example 2: linear elasticity problems}
We consider the linear elasticity problem \eqref{model:elasticity} on the unit square $\Om=(0,1)^2$ with the exact displacement 
$$
\cu= (\sin (\pi x)\sin(\pi y), \sin (\pi x)\sin(\pi y))^T,
$$
and set $f$ and $g$ are chosen corresponding to the above exact solution of \eqref{model:elasticity} with $E=1$ and $\nu=0.4$.
The domain is partitioned by uniform triangles. The level one triangulation $\mathcal{T}_1$ consists of two right triangles, obtained by cutting the unit square with a north-east line. Each triangulation $\mathcal{T}_i$ is refined into a half-sized triangulation uniformly, to get a higher level triangulation $\mathcal{T}_{i+1}$. For this numerical tests, fix the parameters $\rho_1=\rho_2=\gamma=1$.  

\begin{table}[htbp]
  \centering 
    \begin{tabular}{|c|c|c|c|c|c|c|}
\hline
 & $\|\cu - \due\|_{0}$  &  rates     &$\|\due - \Piue\cu\|_{0}$ &   rates      & $\|\cu - \due^*\|_{0}$  &   rates     \\\hline
$\cTh_1$& 9.87E-02 &   -    & 2.90E-02 &   -    & 3.68E-02 &  -\\\hline
$\cTh_2$& 2.37E-02 & 2.06  & 5.65E-03 & 2.36  & 6.83E-03 & 2.43 \\\hline
$\cTh_3$& 3.08E-03 & 2.94  & 3.71E-04 & 3.93  & 3.74E-04 & 4.19 \\\hline
$\cTh_4$& 3.89E-04 & 2.99  & 1.51E-05 & 4.62  & 1.43E-05 & 4.71 \\\hline
$\cTh_5$& 4.87E-05 & 3.00  & 5.17E-07 & 4.87  & 4.79E-07 & 4.90 \\\hline
$\cTh_6$& 6.10E-06 & 3.00  & 1.67E-08 & 4.95  & 1.54E-08 & 4.96 \\\hline
    \end{tabular}%
  \caption{Superconvergence for the elasticity problem with $\alpha=(3,2,2,3)$}
  \label{tab:elasticity1}%
\end{table}%

\begin{table}[htbp]
  \centering 
    \begin{tabular}{|c|c|c|c|c|c|c|}
\hline
 & $\|\cu - \due\|_{0}$  &  rates     &$\|\due - \Piue\cu\|_{0}$  &   rates      & $\|\cu - \due^*\|_{0}$  &   rates     \\\hline
$\cTh_1$& 7.17E-02 &   -    & 2.51E-02 &   -    & 3.47E-02 & - \\\hline
$\cTh_2$& 4.12E-03 & 4.12  & 7.52E-04 & 5.06  & 8.82E-04 & 5.30 \\\hline
$\cTh_3$& 2.69E-04 & 3.94  & 2.27E-05 & 5.05  & 2.28E-05 & 5.27 \\\hline
$\cTh_4$& 1.70E-05 & 3.98  & 5.76E-07 & 5.30  & 5.29E-07 & 5.43 \\\hline
$\cTh_5$& 1.06E-06 & 4.00  & 1.13E-08 & 5.68  & 1.01E-08 & 5.71 \\\hline
    \end{tabular}%
  \caption{Superconvergence for  the elasticity problem with $\alpha=(4,3,3,4)$}
  \label{tab:elasticity2}%
\end{table}%
 
Table \ref{tab:elasticity1} and \ref{tab:elasticity2} list the errors $\|\cu - \due\|_{0}$, $\|\Piue\cu - \due\|_{0}$, $\|\cu - \due^*\|_{0}$ and the corresponding convergence rates of the discontinuous Galerkin formulation \eqref{XGdiv2} with $k\ge n$ for elasticity problem \eqref{model:elasticity}. It is shown that both $\|\Piue\cu - \due\|_{0}$ and $\|\cu - \due^*\|_{0}$ of the discontinuous Galerkin formulation \eqref{XGdiv2} with $k=2$ and $k=3$ converge at the rates 5.00 and 6.00, respectively. This verifies that error estimates in Theorem \ref{th:post2}.
 
We also test the postprocessing scheme \eqref{post:elasticity}  on the formulation \eqref{XGdiv2} with $k<n$, namely,
$$
\alpha=(1,0,0,1)\quad \mbox{and}\quad (2,1,1,2),
$$
where the results are listed in Table \ref{tab:elasticity3} and \ref{tab:elasticity4}, respectively. It shows that postprocessing solution $\due^*$ converges at the same rate as the finite element solution $\due$, which is $k+1$ for the case $k<n$. This implies that there is no such $H({\rm div})$-conforming projection that admits the commuting diagram \eqref{pro:commuting}.
\begin{table}[htbp]
  \centering 
    \begin{tabular}{|c|c|c|c|c|c|c|}
\hline
 & $\|\cu - \due\|_{0}$  &  rates     &$\|\due - \Piue\cu\|_{0}$  &   rates      & $\|\cu - \due^*\|_{0}$  &   rates     \\\hline
$\cTh_2$& 4.25E-01 &- & 2.50E-01 & - & 3.45E-01 & - \\\hline
$\cTh_3$& 2.13E-01 & 0.99  & 1.12E-01 & 1.16  & 1.29E-01 & 1.42 \\\hline
$\cTh_4$& 1.00E-01 & 1.09  & 3.89E-02 & 1.53  & 4.61E-02 & 1.48 \\\hline
$\cTh_5$& 4.83E-02 & 1.05  & 1.41E-02 & 1.46  & 1.84E-02 & 1.33 \\\hline
$\cTh_6$& 2.39E-02 & 1.02  & 6.06E-03 & 1.22  & 8.39E-03 & 1.13 \\\hline
$\cTh_7$& 1.19E-02 & 1.00  & 2.88E-03 & 1.07  & 4.08E-03 & 1.04 \\\hline
$\cTh_8$& 5.96E-03 & 1.00  & 1.42E-03 & 1.02  & 2.03E-03 & 1.01 \\\hline
    \end{tabular}%
  \caption{Superconvergence for  the elasticity problem with $\alpha=(1,0,0,1)$}
  \label{tab:elasticity3}%
\end{table}%

\begin{table}[htbp]
  \centering 
    \begin{tabular}{|c|c|c|c|c|c|c|}
\hline
 & $\|\cu - \due\|_{0}$  &  rates     &$\|\due - \Piue\cu\|_{0}$  &   rates      & $\|\cu - \due^*\|_{0}$  &   rates     \\\hline
$\cTh_1$& 4.74E-01 &    -   & 3.08E-01 &   -    & 4.00E-01 & - \\\hline
$\cTh_2$& 1.11E-01 & 2.09  & 4.11E-02 & 2.90  & 5.27E-02 & 2.92 \\\hline
$\cTh_3$& 2.85E-02 & 1.97  & 7.20E-03 & 2.51  & 8.57E-03 & 2.62 \\\hline
$\cTh_4$& 7.22E-03 & 1.98  & 1.77E-03 & 2.02  & 1.89E-03 & 2.18 \\\hline
$\cTh_5$& 1.82E-03 & 1.99  & 4.56E-04 & 1.96  & 4.64E-04 & 2.03 \\\hline
$\cTh_6$& 4.55E-04 & 2.00  & 1.16E-04 & 1.98  & 1.16E-04 & 2.00 \\\hline
$\cTh_7$& 1.14E-04 & 2.00  & 2.91E-05 & 1.99  & 2.91E-05 & 2.00 \\\hline
    \end{tabular}%
  \caption{Superconvergence for  the elasticity problem with $\alpha=(2,1,1,2)$}
  \label{tab:elasticity4}%
\end{table}%

\section*{Acknowledgements}
The author wishes to  thank the partial support from the Center for Computational Mathematics and Applications, the Pennsylvania State University.
The author would also like to thank Professor Jinchao Xu from the Pennsylvania State University and Professor Jun Hu from the Peking University for their guidance and helpful suggestions pertaining to this work.

%

\bibliographystyle{siamplain}
\bibliography{references}

\begin{thebibliography}{10}

\bibitem{arnold2008finite}
{\sc D.~N. Arnold, G.~Awanou, and R.~Winther}, {\em Finite elements for
  symmetric tensors in three dimensions}, Mathematics of Computation, 77
  (2008), pp.~1229--1251.

\bibitem{arnold1985mixed}
{\sc D.~N. Arnold and F.~Brezzi}, {\em Mixed and nonconforming finite element
  methods: implementation, postprocessing and error estimates}, ESAIM:
  Mathematical Modelling and Numerical Analysis, 19 (1985), pp.~7--32.

\bibitem{arnold2002unified}
{\sc D.~N. Arnold, F.~Brezzi, B.~Cockburn, and L.~D. Marini}, {\em {Unified
  analysis of discontinuous \protect{Galerkin} methods for elliptic problems}},
  SIAM Journal on Numerical Analysis, 39 (2002), pp.~1749--1779.

\bibitem{arnold2002mixed}
{\sc D.~N. Arnold and R.~Winther}, {\em Mixed finite elements for elasticity},
  Numerische Mathematik, 92 (2002), pp.~401--419.

\bibitem{bank2019superconvergent}
{\sc R.~E. Bank and Y.~Li}, {\em Superconvergent recovery of
  \protect{Raviart--Thomas} mixed finite elements on triangular grids}, Journal
  of Scientific Computing, 81 (2019), pp.~1882--1905.

\bibitem{bramble1989local}
{\sc J.~H. Bramble and J.~Xu}, {\em A local post-processing technique for
  improving the accuracy in mixed finite-element approximations}, SIAM journal
  on numerical analysis, 26 (1989), pp.~1267--1275.

\bibitem{Brandts1994Superconvergence}
{\sc J.~H. Brandts}, {\em Superconvergence and a posteriori error estimation
  for triangular mixed finite elements}, Numerische Mathematik, 68 (1994),
  pp.~311--324.

\bibitem{Jan2000Superconvergence}
{\sc J.~H. Brandts}, {\em Superconvergence for triangular order k=1
  \protect{Raviart-Thomas} mixed finite elements and for triangular standard
  quadratic finite element methods}, Applied Numerical Mathematics, 34 (2000),
  pp.~39--58.

\bibitem{chenhigh}
{\sc C.~Chen and Y.~Huang}, {\em High Accuracy Theory of Finite Element
  Methods}, 1995.

\bibitem{Chen2013Superconvergence}
{\sc H.~Chen and B.~Li}, {\em Superconvergence analysis and error expansion for
  the \protect{Wilson} nonconforming finite element}, Numerische Mathematik, 69
  (2013), pp.~125--140.

\bibitem{cockburndevising}
{\sc B.~Cockburn and G.~Fu}, {\em Devising superconvergent \protect{HDG}
  methods with symmetric approximate stresses for linear elasticity by
  \protect{M}-decompositions}, IMA Journal of Numerical Analysis, 38 (2018),
  pp.~566--604.

\bibitem{cockburn2010new}
{\sc B.~Cockburn, J.~Gopalakrishnan, and J.~Guzm{\'a}n}, {\em A new elasticity
  element made for enforcing weak stress symmetry}, Mathematics of computation,
  79 (2010), pp.~1331--1349.

\bibitem{cockburn2009unified}
{\sc B.~Cockburn, J.~Gopalakrishnan, and R.~Lazarov}, {\em {Unified
  hybridization of discontinuous \protect{Galerkin}, mixed, and continuous
  Galerkin methods for second order elliptic problems}}, SIAM Journal on
  Numerical Analysis, 47 (2009), pp.~1319--1365.

\bibitem{cockburn2009superconvergent}
{\sc B.~Cockburn, J.~Guzm{\'a}n, and H.~Wang}, {\em Superconvergent
  discontinuous \protect{Galerkin} methods for second-order elliptic problems},
  Mathematics of Computation, 78 (2009), pp.~1--24.

\bibitem{cockburn2012conditions}
{\sc B.~Cockburn, W.~Qiu, and K.~Shi}, {\em Conditions for superconvergence of
  \protect{HDG} methods for second-order elliptic problems}, Mathematics of
  Computation, 81 (2012), pp.~1327--1353.

\bibitem{cockburn2013superconvergent}
{\sc B.~Cockburn and K.~Shi}, {\em Superconvergent \protect{HDG} methods for
  linear elasticity with weakly symmetric stresses}, IMA Journal of Numerical
  Analysis, 33 (2013), pp.~747--770.

\bibitem{Douglas1989Superconvergence}
{\sc J.~Douglas and J.~Wang}, {\em Superconvergence of mixed finite element
  methods on rectangular domains}, Calcolo, 26 (1989), pp.~121--133.

\bibitem{xu2003super}
{\sc R.~E.~Bank and J.~Xu}, {\em Asymptotically exact a posteriori error
  estimators, part i: Grids with superconvergence}, Journal on Numerical
  Analysis, 41 (2003), pp.~2294--2312.

\bibitem{gastaldi1989sharp}
{\sc L.~Gastaldi and R.~H. Nochetto}, {\em Sharp maximum norm error estimates
  for general mixed finite element approximations to second order elliptic
  equations}, ESAIM: Mathematical Modelling and Numerical
  Analysis-Mod{\'e}lisation Math{\'e}matique et Analyse Num{\'e}rique, 23
  (1989), pp.~103--128.

\bibitem{gopalakrishnan2012second}
{\sc J.~Gopalakrishnan and J.~Guzm{\'a}n}, {\em A second elasticity element
  using the matrix bubble}, IMA Journal of Numerical Analysis, 32 (2012),
  pp.~352--372.

\bibitem{hong2020extended}
{\sc Q.~Hong, J.~Hu, L.~Ma, and J.~Xu}, {\em An extended \protect{Galerkin}
  analysis for linear elasticity with strongly symmetric stress tensor}, arXiv
  preprint arXiv:2002.11664,  (2020).

\bibitem{hong2012discontinuous}
{\sc Q.~Hong, J.~Hu, S.~Shu, and J.~Xu}, {\em A discontinuous
  \protect{Galerkin} method for the fourth-order curl problem}, Journal of
  Computational Mathematics,  (2012), pp.~565--578.

\bibitem{hong2016robust}
{\sc Q.~Hong, J.~Kraus, J.~Xu, and L.~Zikatanov}, {\em A robust multigrid
  method for discontinuous \protect{Galerkin} discretizations of stokes and
  linear elasticity equations}, Numerische Mathematik, 132 (2016), pp.~23--49.

\bibitem{hong2019unified}
{\sc Q.~Hong, F.~Wang, S.~Wu, and J.~Xu}, {\em A unified study of continuous
  and discontinuous \protect{Galerkin} methods}, Science China Mathematics, 62
  (2019), pp.~1--32.

\bibitem{XGHong}
{\sc Q.~Hong, S.~Wu, and J.~Xu}, {\em An \protect{Extended Galerkin Analysis
  for Elliptic Problems}}, arXiv preprint arXiv:1908.08205v2,  (2019).

\bibitem{hong2018uniform}
{\sc Q.~Hong and J.~Xu}, {\em Uniform stability and error analysis for some
  discontinuous galerkin methods}, arXiv preprint arXiv:1805.09670,  (2018).

\bibitem{hu2015finite}
{\sc J.~Hu}, {\em Finite element approximations of symmetric tensors on
  simplicial grids in $\mathbb{R}^n$: the higher order case}, Journal of
  Computational Mathematics, 33 (2015), pp.~1--14.

\bibitem{hu2018optimal}
{\sc J.~Hu, L.~Ma, and R.~Ma}, {\em Optimal superconvergence analysis for the
  \protect{Crouzeix-Raviart and the Morley} elements}, arXiv preprint
  arXiv:1808.09810,  (2018).

\bibitem{Hu2016Superconvergence}
{\sc J.~Hu and R.~Ma}, {\em Superconvergence of both the
  \protect{Crouzeix-Raviart} and \protect{Morley} elements}, Numerische
  Mathematik, 132 (2016), pp.~491--509.

\bibitem{hu2014family}
{\sc J.~Hu and S.~Zhang}, {\em A family of conforming mixed finite elements for
  linear elasticity on triangular grids}, arXiv preprint arXiv:1406.7457,
  (2014).

\bibitem{hu2015family}
{\sc J.~Hu and S.~Zhang}, {\em A family of symmetric mixed finite elements for
  linear elasticity on tetrahedral grids}, Science China Mathematics, 58
  (2015), pp.~297--307.

\bibitem{hu2016finite}
{\sc J.~Hu and S.~Zhang}, {\em Finite element approximations of symmetric
  tensors on simplicial grids in $\mathbb{R}^n$: The lower order case},
  Mathematical Models and Methods in Applied Sciences, 26 (2016),
  pp.~1649--1669.

\bibitem{hu2014finite}
{\sc J.~Hu and S.~Zhang}, {\em Finite element approximations of symmetric
  tensors on simplicial grids in $\mathbb{R}^n$: the lower order case},
  Mathematical Models and Methods in Applied Sciences, 26 (2016),
  pp.~1649--1669.

\bibitem{stenberg1988family}
{\sc R.~Stenberg}, {\em A family of mixed finite elements for the elasticity
  problem}, Numerische Mathematik, 53 (1988), pp.~513--538.

\bibitem{stenberg1991postprocessing}
{\sc R.~Stenberg}, {\em Postprocessing schemes for some mixed finite elements},
  ESAIM: Mathematical Modelling and Numerical Analysis-Mod{\'e}lisation
  Math{\'e}matique et Analyse Num{\'e}rique, 25 (1991), pp.~151--167.

\bibitem{MixLDGWu}
{\sc F.~Wang, S.~Wu, and J.~Xu}, {\em A mixed discontinuous \protect{Galerkin}
  method for linear elasticity with strongly imposed symmetry}, arXiv preprint
  arXiv:1902.08717,  (2019).

\bibitem{xie2009numerical}
{\sc Z.~Xie, Z.~Zhang, and Z.~Zhang}, {\em A numerical study of uniform
  superconvergence of \protect{LDG} method for solving singularly perturbed
  problems}, Journal of Computational Mathematics,  (2009), pp.~280--298.

\bibitem{xu1992iterative}
{\sc J.~Xu}, {\em Iterative methods by space decomposition and subspace
  correction}, SIAM review, 34 (1992), pp.~581--613.

\end{thebibliography}
\end{document}